\documentclass[12pt]{amsart}
\usepackage{a4wide}
\usepackage{amsmath,amsfonts,amsthm}
\usepackage{upref}
\usepackage{amssymb}
\usepackage[T1]{fontenc}
\usepackage[utf8]{inputenc}
\usepackage{enumerate}
\usepackage{cleveref}

\input{diagrams.tex}
\diagramstyle[scriptlabels,height=8mm,width=8mm]

\newtheorem{theorem}{Theorem}[section]
\newtheorem*{property*}{Property}
\newtheorem*{definition}{Definition}
\newtheorem{example}[theorem]{Example}
\newtheorem{proposition}[theorem]{Proposition}
\newtheorem{corollary}[theorem]{Corollary}
\newtheorem{lemma}[theorem]{Lemma}

\theoremstyle{remark}

\newtheorem*{example*}{Example}
\newtheorem*{remark*}{Remark}

\begin{document}
\newcommand{\set}[1]{\left\{{#1}\right\}}
\newcommand{\Aut}{\operatorname{Aut}}
\newcommand{\dbar}{\overline{\partial}}
\newcommand{\paren}[1]{\left( {#1}\right)}
\newcommand{\expi}[1]{e^{2\pi i {#1}}}
\newcommand{\I}{[0,1]}
\newcommand{\R}{\mathbb{R}}
\newcommand{\Z}{\mathbb{Z}}
\newcommand{\C}{\mathbb{C}}
\newcommand{\B}{\mathbb{B}}
\newcommand{\N}{\mathbb{N}}
\newcommand{\hol}{\mathcal{O}}
\newcommand{\ci}{C^{\infty}}

\renewcommand{\div}{\operatorname{div}}
\newcommand{\SU}{\operatorname{SU}}
\newcommand{\SL}{\operatorname{SL}}
\newcommand{\GL}{\operatorname{GL}}
\newcommand{\Hol}{\operatorname{Hol}}
\newcommand{\Cont}{\operatorname{Cont}}
\newcommand{\Lie}{\operatorname{Lie}}
\newcommand{\Span}{\operatorname{Span}}
\newcommand{\Vect}{\operatorname{Vec}}
\title[An Oka Principle for a Transitivity Property]{An Oka Principle for a Parametric Infinite Transitivity Property}
\author{Frank Kutzschebauch}
\address{Dpt. of Mathematics\\
University of Bern\\
Sidlerstrasse 5, CH-3012. Bern, Switzerland. }
\email{frank.kutzschebauch@math.unibe.ch}
\author{Alexandre Ramos-Peon}
\address{Dpt. of Mathematics\\
University of Bern\\
Sidlerstrasse 5, CH-3012. Bern, Switzerland. }
\email{alexandre.ramos@math.unibe.ch}
\date{\today}
{\renewcommand{\thefootnote}{} \footnotetext{

Both authors partially supported by Schweizerischer Nationalfond Grant  200021-140235/1
}
\renewcommand{\thefootnote}{\arabic{footnote}}

\begin{abstract}
It is an elementary fact that the action by holomorphic automorphisms on $\C^n$ is infinitely transitive, 
i.e., $m$-transitive for any $m\in \N$. The same holds on any Stein manifold with the holomorphic density property $X$.
We study a parametrized case: we consider $m$ points on $X$ parametrized by a Stein manifold $W$, 
and seek a family of automorphisms of $X$, parametrized by $W$, putting them into a standard form which does not depend on the parameter.
This general transitivity is shown to enjoy an Oka principle, to the effect that the obstruction to a holomorphic solution
is of a purely topological nature. In the presence of a volume form and of a corresponding density property, similar results for volume-preserving automorphisms are obtained.
\end{abstract}
\maketitle
\section{Introduction}
Let $X$ and $W$ be complex manifolds. 
Let  $Y_{X,N}$ be  the configuration space of ordered $N$-tuples of points in $X$: $Y_{X,N}=X^N\setminus \Delta$, 
where 
\[
\Delta=\set{(z^{1},\dots,z^{N})\in X^N;z^{i} = z^{j} \text{ for some } i  \ne j} 
\]
is the diagonal. Consider a holomorphic map $x:W\to Y_{X,N}$, that is, 
$N$ holomorphic maps $x^j:W\to X$ such that for each $w\in W$, the $N$ points $x^1(w),\dots,x^N(w)$ are pairwise distinct. Interpreting $x:W\to Y_{X,N}$ as a parametrized collection of points, 
 the following property can be thought of as a strong type of $N$-transitivity. 
\begin{definition}
Assume that $\Aut(X)$ acts $N$-transitively on $X$. Fix $N$ pairwise distinct points $z^1,\dots,z^N$ in X. 
We say that the parametrized points $x^1,\dots,x^N$ are \emph{simultaneously standardizable} if there exists a ``parametrized automorphism''  $\alpha\in\Aut_{W}(X)$, where 
\[\Aut_{W}(X)=\{ \alpha\in\Aut(W\times X); \alpha(w,z)=(w,\alpha^w(z))\},\]
with 
\[\alpha^w(x^j(w))=z^j\]
for all $w\in W$ and $j=1,\dots,N$. 
\end{definition}
By the transitivity assumption, the definition does not depend on the choice of the $z^j$'s. 
Note also that in this paper's context, automorphisms are always \emph{holomorphic automorphisms} in the sense that we are not interested in the algebraic category.

This notion was introduced by the first author and S. Lodin in \cite{KL13},  
where it is shown that for $X=\C^n$ and $W=\C^k$, if $k<n-1$, then any collection of parametrized points $W\to Y_{X,N}$ is simultaneously standardizable. 
Our main result is the following.
\begin{theorem}
\label{thm:main}
Let $W$ be a Stein manifold and $X$ a Stein manifold with the holomorphic density property. Let $N$ be a natural number and 
$x:W\to Y_{X,N}$ be a holomorphic map. Then the parametrized points $x^{1},\dots,x^{N}$ are simultaneously standardizable by an automorphism 
lying in the path-connected component of the identity $(\Aut_W(X))^0$ of $\Aut_W(X)$ if and only if $x$ is null-homotopic.
\end{theorem}

Since being null-homotopic is a purely topological condition, Theorem~\ref{thm:main} is an Oka principle for a strong form of parametric infinite transitivity.
In the particular case when $W=\C^k$ and $X=\C^n$, any map $W\to Y_{X,N}$ is null-homotopic, so we recover the result of \cite{KL13}, without any restrictions on the dimension of $W$. Moreover, 
Theorem~\ref{thm:main} reduces the problem of simultaneous standardization of parametrized points in $\C^n$ 
by  automorphisms in $\Aut_W(\C^n)$ (not the connected component!) to a purely topological problem as explained in Section~\ref{last}, Corollary \ref{hom-coro}. 
This is a (slightly different) Oka principle for a strong form of parametric $N$-transitivity.

We are also able to prove a similar result when $X$ is a manifold with the $\omega$-volume density property (instead of the density property) under an additional topological assumption.
\begin{theorem}\label{thm:VDP}
Let $X$ be a Stein manifold with a volume form $\omega$ which satisfies the $\omega$-volume density property. 
Assume $X$ has dimension greater than $1$ and that is contractible.
Then similarly, for any natural number $N$ a holomorphically parametrized collection
 of points $x:W\to Y_{X,N}$ can be simultaneously standardized by a volume-preserving automorphism lying in the path-connected component of the identity
$(\Aut_W(X,\omega))^0$ of $\Aut_W(X,\omega)$ if and only if $x$ is null-homotopic.
\end{theorem}
The dimension assumption is obviously necessary, as will be seen below. However we do not know if contractibility can be relaxed for the conclusion to hold. 

\begin{remark*}
It is presently still unknown whether a contractible Stein manifold with the volume density property has to be biholomorphic to  $\C^n$. The authors believe that there
are plenty of them not biholomorphic to $\C^n$, but that the tools for distinguishing them biholomorphically from $\C^n$ have yet to be developed. 
Concrete examples can be found in \cite{KK-volume}. 
For instance the affine algebraic submanifold of $\C^6$ given by the equation $uv = x+ x^2y + s^2 + t^3$ is such an example (for contractibility see the appendix of \cite{KK-Zeit}). 
Another prominent example is the Koras-Russell cubic threefold, see \cite{L}. 
Furthermore, candidate counterexamples to the holomorphic version of the Zariski cancellation problem, of the form $uv = f(z)$ where $f \in \hol (\C^n)$ 
generates the ideal of functions vanishing on a non straightenable embedded $\C^{n-1}$, are of that kind: see \cite{RP}.
\end{remark*}

In what follows the dependence of an automorphism on a parameter is always understood to be a \emph{holomorphic} dependence as just described. 
A homotopy connecting two maps $f_0$ and $f_1$ between any two complex manifolds $W\to X$ is only assumed to be a \emph{continuous} function $f:W\times\I\to X$.
If each $f_t$ is holomorphic, we speak of a \emph{homotopy through holomorphic maps}, 
and if furthermore the function is $\mathcal{C}^k$ (resp. $\mathcal{C}^\infty$), it is a $\mathcal{C}^k$ (resp. \emph{smooth}) \emph{homotopy} between $f_0$ and $f_1$.
Finally if the variable $t$ is allowed to vary in a complex disc $D_r\subset \C \ (r>1)$, and $f$ is holomorphic, we speak of an \emph{analytic} homotopy.

The paper is organized as follows. 
In Section~\ref{sect1} we recall the definition of manifolds with the density property and we prove a general parametric version of the Anders\'en-Lempert theorem.
In Section~\ref{sect2.1} we recall the $\omega$-volume density property and discuss the approximation of local holomorphic phase flows by volume-preserving automorphisms in the parametric case, which turns out to be more elusive.
In Section~\ref{sect2} we establish that $Y_{X,N}$ is elliptic in Gromov's sense and hence an Oka-Grauert-Gromov h-principle applies to maps $W\to Y_{X,N}$; 
this will allow us to use the Andersén-Lempert theorem.
Section~\ref{sect4} contains the details of the proof of Theorem~\ref{thm:VDP}, from which Theorem~\ref{thm:main} also follows.
The idea is to define a countable sequence of automorphisms, each of which maps $x$ closer to some constant $\hat{x}$ on a larger set, which converges to the desired standardization.
In Section~\ref{last} we make explicit a homotopy-theoretical point of view, and prove a version of Grauert's Oka principle for principal bundles. 
Finally we consider two examples which illustrate cases in which the topological obstruction of Corollary \ref{hom-coro} vanishes, and does not vanish, respectively. 
\section{Density Properties and the Anders\'en-Lempert Theorem}\label{sect1}
Let $X$ be a complex manifold, $U\subset X$ an open set,  and $k\in \N\cup\{ \infty\}$. 
A $\mathcal{C}^k$ \emph{isotopy of injective holomorphic maps} is a $\mathcal{C}^k$ map $F:U\times \I\to X$ such that for each
fixed $t\in\I$, the map $F_t:U\to X$ is an injective holomorphic map. 
The main theorem in \cite{FR} states that given a $\mathcal{C}^2$ isotopy of injective holomorphic maps $F_t:U_0\to U_t$ between Runge domains\footnote{
Recall that a pseudoconvex (Stein) domain $U\subset X$ is Runge in $X$ if holomorphic functions on $U$ can be approximated, uniformly on compacts of $U$, by functions holomorphic on $X$. Here all Runge domains are taken to be Stein.} in $\C^n$ such that $F_0$ is the identity,
 then all the maps $F_t$ can be approximated uniformly on compacts by automorphisms of $\C^n$. 
In the same paper, approximation ``near polynomially convex sets'' is proved, in \cite{F94} the required regularity is shown to be $\mathcal{C}^0$, 
and in \cite{K-ALwithparams} a parametric version is shown to hold 
(see also \cite{F-noncritical}, 
where it is used to prove an approximation result for holomorphic submersions). Combining this we obtain:
\begin{theorem}[Anders\'en-Lempert Theorem]
Let $n\geq 2$ and $U$ be an open set in $\C^k\times\C^n$.
Let $F$ be a $\mathcal{C}^p\ (p\geq 0)$ isotopy of injective holomorphic maps from $U$ into $\C^k\times\C^n$ of the form
\begin{equation}
F_t(w,z)=(w,F_t^w(z)),\quad (w,z)\in U,\quad \text{ and } F_0^w=id. \tag{$\star$}
\end{equation}
Suppose $K\subset U$ is a compact polynomially convex subset of $\C^k\times\C^n$, and  
assume that $F_t(K)$ is polynomially convex in $\C^k\times\C^n$ for each $t\in \I$.   Then for all $t\in\I$,
$F_t$ can be approximated uniformly on $K$  (in the $\mathcal{C}^p$ norm) by automorphisms $\alpha_t\in\Aut_{\C^k}(\C^n)$; moreover
$\alpha_t$ depends smoothly on $t$, and $\alpha_0$ can be chosen to be the identity.
\end{theorem}
We refer to the cited references for a complete proof of this theorem. 
However it is convenient to introduce here the ideas involved, which are best understood in the language of vector fields and their flows. 
In what follows and in the rest of this paper we will only consider holomorphic vector fields, 
that is, sections of the bundle $T^{1,0}X$; we implicitly identify $T^{1,0}X$ with $TX$. 

The isotopy can be interpreted as the flow of a time-dependent vector field, which has no component in the $w$ direction. 
The polynomial convexity is used to construct a neighborhood of $K$ with the property that all of its images under the isotopy are Runge.
The Runge property is then used to approximate this time-depending field at finitely many instants by time-independent global fields (without components in the $w$ direction).
The main point is that each of these global fields is approximated by a finite sum of \emph{complete} fields,
that is, fields for which the flow is assumed to exist for all complex times and initial conditions.
Then the composition of these flows provide automorphisms depending on $w$. A version of this was proven in the early 90's by 
E. Anders\'en and L. Lempert \cite{AL}; the association of their names to this theorem has since then been established in the literature.

In \cite{Varolin1} and \cite{Varolin2}
D. Varolin introduced a class of manifolds where the ``main point'' above holds. 
Namely, a complex manifold $X$ is said to have the (holomorphic) \emph{density property} or DP
if the Lie algebra generated by the complete vector fields is dense in the algebra of all vector fields on $X$ in the compact-open topology. 
The relevant fact 
is that the flow of a vector field on such a manifold can be approximated by flows of complete fields. 
In a recent paper by T. Ritter \cite{Tyson} there is a detailed proof of a general version of the Anders\'en-Lempert theorem for manifolds with the density property, where polynomial convexity must be replaced 
by $\hol(X)$-convexity: given a complex manifold $X$, a compact $K\subset X$ is said to be $\hol(X)$-convex if $K=\widehat{K}_{\hol(X)}$ where
\[
\widehat{K}_{\hol(X)}=\set{x\in X; |f(x)|\leq \sup_{y\in K}|f(y)|\quad \forall f\in\hol(X)}.
\]
If we further allow the maps $F_t$ to depend holomorphically on a parameter $w$ in a Stein manifold $W$, then we can closely follow the proof in \cite{Tyson} by carrying a parameter. 
The only apparent difficulty arises when the density property is used to construct a vector field in the Lie algebra generated by complete fields, 
for the holomorphic dependence of these new fields on $w$ is not obvious. However Lemma 3.5 in \cite{Varolin1} shows precisely that if $V_w$ is a vector field on $X$ depending 
holomorphically on a Stein parameter $w$, then $V_w$ can be approximated locally uniformly on $W\times X$ by Lie combinations of complete vector fields which depend holomorphically on the parameter. 
This proves the following parametric version of the Anders\'en-Lempert theorem in manifolds with the density property.
\begin{theorem}
\label{AL}
Let $W$ be a Stein manifold and $X$ a Stein manifold with the DP. 
Let $U\subset W\times X$ be an open set and $F_t:U\to W\times X$ be a  smooth  isotopy of injective holomorphic maps of the form $(\star)$.
Suppose $K\subset U$ is a compact set such that $F_t(K)$ is $\hol(W\times X)$-convex for each $t\in \I$.  Then for all $t\in\I$,
$F_t$ can be approximated uniformly on $K$ (with respect to any distance function on $X$) by automorphisms $\alpha_t\in\Aut_{W}(X)$ which depend smoothly on $t$, and moreover we can choose $\alpha_0=id$.
\end{theorem}
\subsection{\textit{Volume density property}}\label{sect2.1}
Let now $X$ be a complex manifold equipped with a holomorphic volume form, that is, a nowhere vanishing holomorphic $n$-form $\omega$, and denote by $\Vect(X,\omega)$ the Lie algebra of 
vector fields $\Theta$ on $X$ which are $\omega$-divergence free, that is, such that $L_\Theta \omega=0$.
Here $L$ denotes the Lie derivative; 
by the Cartan formula, this is equivalent to saying that the $n-1$  form $i_{\Theta}\omega$, the contraction of $\omega$ by $\Theta$, is closed.
Following  \cite{Varolin1} we say that $X$ has the $\omega$-volume density property (or simply $\omega$-VDP) if the Lie algebra generated by the complete $\omega$-divergence free fields is dense in $\Vect(X,\omega)$. 
We will also consider $\Aut(X,\omega)$ to be the automorphisms $\alpha$ preserving $\omega$ (i.e. $\alpha^*\omega=\omega$), which we call volume-preserving. 
Further, we denote by $\Aut_W(X,\omega)$ the set of parametric volume-preserving automorphisms of $X$, 
i.e. automorphisms of $W\times X$ of the form $\alpha(w,z)=(w,\alpha^w(z))$ with each $\alpha^w$ volume-preserving.

In the non-parametric case (when $W$ is a singleton), one finds for example in \cite{KK-state} 
a volume-preserving analogue of the Anders\'en-Lempert theorem on Stein manifolds with the $\omega$-volume density property. 
It was noticed in \cite{Erratum} that
the approximation cannot be on arbitrary $\hol(X)$-compact subsets $K$ of $X$: there are topological obstructions, and the vanishing of $\mathrm{H}^{n-1}_{DR}(K)$ is shown to be sufficient.
In fact the crucial condition is an extension property, which we state immediately in the parametric case.
Let $W$ be as before a Stein manifold and denote by $\pi_W:W\times X \to W$ the projection, and for a subset $U$ of $W\times X$, denote the ``$w$-slices'' by
\[U_w=(\set{w}\times X)\cap U,\]
and its projection to $W$ by
\[
U'=\pi_W(U).
\]
Let $\Omega^k(X)$ (resp. $Z^k(X)$) be the space of holomorphic $k$-forms on $X$ (resp. closed forms).
We want to consider holomorphic mappings of the form 
\begin{equation}\label{betaw}
w\mapsto \beta^w\in\Omega^k(U_w) ,\quad w\in U'=\pi_W(U).
\end{equation}
For this consider the pullback of the bundle $\Omega^k(X)$ by the projection $\pi_X:W\times X\to X$
and denote this bundle over $W\times X$ by $\Omega_W^k(X)$. 
Its global sections are forms on $W\times X$ which locally are of type $\sum_{I}h(w,z)dz_{I}$.
Denote the local sections on $U\subset W\times X$ by $\Omega_W^k(U)$; they define a
coherent sheaf on $W\times X$, 
and we identify a local section $\beta\in\Omega_W^k(U)$ to a holomorphic mapping as in equation (\ref{betaw}).
We can define $\Vect_W(X)$ analogously.

Because divergence free vector fields do not form an analytic subsheaf of $\Vect(X)$, 
the difficulty of proving an obvious analogue of Theorem~\ref{AL} lies in a Runge-type approximation of a locally defined divergence free vector field by a global divergence free field. 
\begin{definition}
Let $U\subset W\times X$. We say a local section $\Theta\in\Vect_W(U,\omega)$, that is, a holomorphic map
\[
w\mapsto \Theta^w\in\Vect(U_w,\omega),\quad w\in U'
\]
is \emph{globally approximable} if there exists a global section $\widehat\Theta\in\Vect_W(W\times X,\omega)$, that is 
a holomorphic map
\[
w\mapsto \widehat\Theta^w\in\Vect(X,\omega),\quad w\in W
\]
 approximating
$\Theta$ uniformly on compacts of $U$.
\end{definition}
Next we explain what our sufficient condition is.
Assume that $X$ is a Stein manifold with the $\omega$-VDP, and let $U\subset W\times X$ be open.
Let $F_t:U\to W\times X$ be a smooth isotopy of injective, volume-preserving holomorphic maps of the form $(\star)$.
Consider now the $\omega$-divergence free vector fields
\[
\Theta_t^w=\frac{dF_s^w}{ds}\bigg|_{s=t}\circ (F_t^w)^{-1}.
\]
If each $\Theta_t$ can be globally approximated in the sense just defined, with smooth dependence on $t$, then the $\omega$-VDP
can be used exactly as in the proof of  Theorem~\ref{AL} to show that 
 each $F_t$ can be approximated uniformly on compacts of $U$  by volume-preserving automorphisms $\alpha_t\in\Aut_W(X,\omega)$ which depend smoothly on $t$, with $\alpha_0=id$.

We will now give two instances where such a global approximation is possible, both of which will be used below. We fix from now on a distance function $d$ on $X$.
\begin{proposition}\label{isotopy1}
Let $W$ and $X$ be Stein manifolds and assume that $X$ has an $\omega$-VDP.
Let $f:W\times \I\to X$ be a  smooth homotopy through holomorphic maps between $f_0$ and $f_1$. If $L\subset W$ is a $\hol(W)$-convex compact, 
then given $\epsilon>0$ there exists $A_t\in\Aut_W(X,\omega)$, with $A_0=id$, depending smoothly on $t$, such that
\[
d( A_t^w\circ f_0(w),f_t(w))<\epsilon\quad \forall (w,t)\in L\times\I.\qedhere
\]
\end{proposition}
Observe that a similar result holds for maps into $X^N$ and therefore, in the notation of the introduction, 
into $Y_{X,N}$ (see the proof of Corollary \ref{stepchange} with the notation preceding Lemma \ref{spanning}).
Furthermore, note that if $X$ has the DP, then the proof below can be considerably simplified to obtain the same result with $A_t\in\Aut_W(X)$ only. 
\begin{proof}
We claim that there is a suitable neighborhood $U\subset W\times X$ of the graph
\[
\Gamma_{L}(f_0)=\set{(w,f_0(w)); w\in L},
\]
with contractible fibers $U_w$, and on it an isotopy of injective volume-preserving holomorphic maps $F$ of the form ($\star$) extending the definition of $f_t$, i.e.
\[
F_t^w(f_0(w))=f_t(w)\quad \forall (w,t)\in U'\times\I.
\]
Since $f$ can be thought of as a section of the trivial holomorphic fibration $W\times X$, the pullback by $f$ of the normal bundle on $W\times X$ is trivial over $L$. 
Hence for each $w\in L$, there is a (contractible) coordinate neighborhood $U_0(w)\subset X$ of $f_0(w)$ with chart 
\[\phi^w_0:U_0(w)\to B\subset \C^n\]
mapping $f_0(w)$ to $0$, depending holomorphically on $w$, and such that the restriction of $\omega$ to $U_0(w)$ is $(\phi^w_0)^*(\tilde{\omega}_w)$,
where $\tilde\omega_w$ is some volume form on $B$: 
\[\tilde\omega_w(z)=g(z,w)dz_1\wedge\cdots\wedge dz_n,\quad g\in\hol(B\times L).
\]
By the Moser trick and compactness of $L$ we may shrink $B$ and $U_0(w)$ in order that for all $w\in L$,
\[
\tilde\omega_w(z)=g(0,w)dz_1\wedge\dots \wedge dz_n\quad \forall z\in B.
\]
Note that this $z$-independent formula for $\tilde\omega_w$ holds with respect to a possibly different family of coordinate charts, which we still denote by $\phi^w_0$.
Again by compactness there exist coordinate neighborhoods $U_0(w),U_{t_1}(w),\dots,U_{1}(w)$ of $f_0(w),f_{t_1}(w),\dots,f_1(w)$ respectively, 
each of which are equivalent to $B$ with a constant volume form, covering 
$
\set{f_t(w);t\in \I}\subset X
$
for each $w\in L$. 
On $U_0(w)$ we define, for each $t\in[0,\tau_0(w))$ (where $\tau_0(w)$ is such that $f_t(w)\in U_0(w)$ for all $t<\tau_0(w)$), 
a $\omega$-divergence free field $\Theta_t^{(0)}(w)$ depending holomorphically on $w$ by pulling back the field on $B$ which is constantly equal to 
\[\frac{d}{d s}\bigg|_{s=t}\phi_0^w\circ f_s(w).\]
Similarly on $U_{t_1}(w)$ there is such a family of fields $\Theta_t^{(1)}(w)$ ($\tau_1'(w)<t<\tau_1(w))$, 
so by using a suitable smooth cut-off function $\chi^w(t)$ one can further define on $U_0(w)\cup U_{t_1}(w)$ the fields
\[\chi^w(t)\Theta_t^{(0)}(w)+(1-\chi^w(t))\Theta_t^{(1)}(w),\quad t\in[0,\tau_1(w)),\]
which are still divergence free and restrict to $\frac{d}{d s}\big|_{s=t}f_s(w)$.
For fixed $w$, a small enough neighborhood of $f_0(w)$ will flow entirely inside of $U_0(w)\cup U_{t_1}(w)$ under the flow of the 
above time-dependent vector field. The claim is proved by repeating this construction until the last intersection with $U_1(w)$:
we get a neighborhood $U$ of $\Gamma_{L}(f_0)$ and the desired isotopy $F_t$ consists of the time-$t$ maps of the flow of the described time-dependent field.
Note that $U$ can also be chosen with the property that $F_t(U)$ is Runge for all $t\in\I$: since
$F_t(\Gamma_{W}(f_0))=\Gamma_{W}(f_t)$ is an analytic set in $W\times X$,
its $\hol(W\times X)$-convexity easily follows from the Cartan extension theorem on Stein manifolds; proceed then as in the proof of Lemma 2.2 in \cite{FR}.

We have seen that it suffices to show that each field
\[
\Theta_t^w=\frac{dF_s^w}{ds}\bigg|_{s=t}\circ (F_t^w)^{-1}
\]
can be globally approximated with smooth dependence on $t$.
Let $\eta_t\in Z_W^{n-1}(F_t(U))$ be a section of $Z^{n-1}_W(X)$ defined by
\[
\eta_t^w=i_{\Theta_t^w}\omega\in Z^{n-1}(F_t(U_w)),\quad w\in U'.
\]
The set $F_t(U)$ is fiberwise contractible, and in fact there is a contraction of each $F_t^w(U_w)$ depending holomorphically on $w$, hence Poincar\'e's lemma 
gives an explicit section
$\beta_t\in\Omega_W^{n-2}(F_t(U))$ satisfying
\[
d\beta_t^w=\eta_t^w. 
\]
By the Runge property and Cartan's theorem A, the forms $\beta_t$ can be approximated by global sections in this coherent sheaf, i.e., there a holomorphic map $\widehat{\beta}_t:W \to \Omega^{n-2}(X)$
approximating $\beta_t$, and by the Cauchy estimates we can even ensure that $d\beta_t$ approximates $d\widehat{\beta}_t$. 
In fact, it is classical that these approximations can be taken to be smooth on $t$ (see \cite{Bungart}).
There is a unique vector field $\widehat{\Theta}_t^w$ given by the duality $i_{\widehat{\Theta}_t^w}\omega=d\widehat{\beta}_t^w$ (since $\omega$ is non-degenerate), which is then divergence free. 
By standard theory of differential equations, it approximates $\Theta^w_t$. 
\end{proof}
\begin{proposition}\label{isotopy2}
Let $W$ be Stein and suppose that $X$ is a contractible Stein manifold with an $\omega$-VDP and of dimension at least two.
Suppose that the compact $K\subset W\times X$ has the following form: $K'=\pi_W(K)$ is $\hol(W)$-convex, and there is a $\hol(W)$-convex compact $L'\subseteq K'$ such that 
\[
K=\Gamma_{K'}(g)\cup (L'\times S),
\]
where $S$ is a compact $\hol(X)$-convex subset of $X$ and $g:W\to X$ is holomorphic.
Let $F_t:U\to W\times X$ be an isotopy of injective volume-preserving holomorphic maps of the form $(\star)$ defined on a neighborhood $U$ of $K$ which has the property that $U_w=X$ for all $w\in V$ where $V$ is a neighborhood of $L'$.
Then for any $\epsilon>0$ there exists $A_t\in\Aut_W(X,\omega)$, with $A_0=id$, depending smoothly on $t$, such that
\[
d( A_t^w(z),F_t^w(z))<\epsilon\quad \forall (w,z,t)\in K\times \I.\qedhere
\]
\end{proposition}
\begin{proof}
We may assume without loss of generality that $U$ is a neighborhood of $K$ of the form 
$U=A\cup B$ where $A$ is a fiberwise contractible neighborhood of $\Gamma_{U'}(g)$ and $B=V\times X$, such that 
that $F_t(U)$ is Runge (because $K$ is easily seen to be $\hol(W\times X)$-convex: see then argument in the previous proof) 
and $F_t(A)$ has contractible fibers for all $t\in \I$.
Define as before\[
\Theta_t^w=\frac{dF_s^w}{ds}\bigg|_{s=t}\circ (F_t^w)^{-1}\ \text{ and }\ \eta_t^w=i_{\Theta_t^w}\omega\in Z^{n-1}(F_t(U_w)),\quad w\in U'.
\]
By the Poincar\'e lemma and the contractibility of $X$ and of the fibers of $F_t(A)$, there are local sections
$\beta_{A,t}\in\Omega^{n-2}_W(F_t(A))$ and $\beta_{B,t}\in\Omega_W^{n-2}(F_t(B))$
such that $d\beta_{A,t}^w=\eta_t^w$ on $F_t(A)$ and $d\beta_{B,t}^w=\eta_t^w$ on $F_t(B)$.
It now suffices to find a single family $\beta_t\in\Omega_W^{n-2}(F_t(U))$ depending smoothly on $t$ and satisfying $d\beta_t^w=\eta_t^w$ for all $w\in U'$: we would then conclude as in the previous proof.

For simplicity fix $t=0$.
The $n-2$ form $\beta_A-\beta_B$ is closed on $A\cap B$, a set with contractible fibers, so again $\beta_A^w-\beta_B^w=d\delta_{AB}^w$ for a section $\delta_{AB}\in\Omega_W^{n-3}(A\cap B)$.
Consider the covering $\mathcal{U}=\set{A,B}$ of $U$ and let $\check{\mathrm{H}}^1(\mathcal{U},\Omega_W^{n-3}(U))$ be
the first \v Cech cohomology group of the covering $\mathcal{U}$ with values in the sheaf $\Omega_W^{n-3}(U)$. 
By Cartan's theorem B, the sheaf  $\Omega_W^{n-3}(U)$ is acyclic on $A,B$ and $A\cap B$, 
so by Leray's theorem 
\[
\check{\mathrm{H}}^1(\mathcal{U},\Omega_W^{n-3}(U))=\mathrm{H}^1(U,\Omega_W^{n-3}(U)).
\]
But again the right-hand side is trivial.
The vanishing of the \v Cech cohomology group yields a splitting
\[
\delta_{AB}^w=\delta_B^w-\delta_A^w, 
\]
where $\delta_A$ (resp. $\delta_B$) is a section of $\Omega_W^{n-3}(X)$ on $A$ (resp. $B$). 
Now since
\[
\beta_A^w+d\delta_A^w=\beta_B^w+d\delta_B^w\quad \forall w\in V,
\]
the ``glueing'' property of the sheaf gives the desired $\beta^w$.
In the case that $n=2$, replace the above formula by $\beta_A^w=c+\beta_B^w$, where $c$ is a constant.

When $t$ is allowed to vary smoothly in $\I$, the form $\delta_{AB,t}$ above can be chosen depending smoothly on $t$.
We can consider the sheaf of smooth maps from $\I$ into $\Omega_W^{n-3}$,
whose cohomology is shown in \cite{Bungart} to vanish (as in Cartan's theorem B), 
so the argument above carries to this new sheaf and we obtain a smoothly depending family of sections $\beta_t\in\Omega_W^{n-2}(F_t(U))$ as desired.
\end{proof}
\section{The Oka Property}\label{sect2}
Stein manifolds with the density property are of interest not only in view of the Anders\'en-Lempert approximation described previously, but also because
they enjoy a flexibility property now referred in the literature as the ``Oka property''. 

There are several equivalent characterizations of Oka-Forstneri\v c manifolds; the survey \cite{survey} gives a detailed account. 
For our purposes, we define $Y$ to be an \emph{Oka-Forstneri\v c  manifold} if it enjoys the following property, called in the literature the \emph{Basic Oka Property with approximation and interpolation}:

\begin{property*}
Let $T$ be a closed complex submanifold of a Stein manifold $S$, and $K$ be a $\hol(S)$-convex compact subset of $S$.
Let $f:S\to Y$ be a continuous map such that  $f$ is holomorphic in a neighborhood of $K$, and $f$ is holomorphic on $T$.
Then there is (a homotopy joining $f$ to) some holomorphic $g:S\to Y$ such that $g=f$ on $T$ and $g$ is uniformly close to $f$ on $K$. 
\end{property*}
\begin{theorem}\label{YisOka}
If $X$ is a Stein manifold with the DP or a Stein manifold of dimension greater than one with the $\omega$-VDP, then $Y_{X,N}$ is an Oka-Forstneri\v c manifold for any $N$.
\end{theorem}
We remark that if $X$ has the $\omega$-VDP and dimension $1$, then it must be either $(\C,dz)$ or $(\C^*,z^{-1}dz)$.
If $N\geq 4$, or in the case $X=\C^*$, then $X^N\setminus \Delta$ is a projective space with too many hyperplanes removed, and this cannot be Oka-Forstneri\v c 
by Theorem 3.1 in \cite{Hanysz}. The same result shows that if $X=\C$ and $N=2$ or $3$, then $X^N\setminus \Delta$ is indeed Oka-Forstneri\v c.

Theorem~\ref{YisOka} will be deduced from the following lemma. First we introduce some more notation. Let $Y=Y_{X,N}$ and 
define the linear map
\[
\oplus:\Vect(X)\to \Vect(Y)
\]
as follows: for each $V\in\Vect(X)$ let $\oplus V \in\Vect(X^N)$ be the vector field in $X^N$ defined by  $\oplus V(z^1,\dots,z^N)=(V(z^1),\dots,V(z^N))\in T_{(z^1,\dots,z^N)}X^N$.
Clearly $\oplus V\in\Vect(Y)$, and since in fact this field is tangent to $\Delta$, the image of a point in $Y$ under the flow of $\oplus V$ remains in $Y$. 
It is clear that $\oplus$ restricts to a map between complete fields. Similarly, we have an obvious map
\[
\oplus: \Aut(X)\to \Aut(Y).
\]
\begin{lemma}\label{spanning}
Let $X$ be a Stein manifold with the DP (resp. with the $\omega$-VDP and dimension greater than one) and let $N\geq 1$.
Then there exist complete (resp. and divergence free) vector fields 
$V_1,\dots,V_m$ on $X$ such that
\[T_{y}Y=\Span\{\oplus V_j (y)\}_j\quad \forall y\in Y=Y_{X,N}.\]
\end{lemma}
In particular, the statement for $N=1$ holds. That case is treated in \cite{KK-state}. We adapt that proof to this more general situation.
\begin{proof}
We give the proof for a manifold with the DP and only give indications of the modifications required for the $\omega$-VDP case.
Let $x^1,\dots,x^N\in X$ be $N$ pairwise distinct points in $X$. Since $X$ is Stein we can pick a Runge open set around $\{x^1\}\cup\cdots\cup\{x^N\}$ of the form
$U=\dot\bigcup_{j=1}^N U^j$, so small that a chart $U^j\to \C^n$ exists for each $j$, where $n$ is the dimension of $X$
(as in the proof of Proposition~\ref{isotopy1}  $\omega|_{U^j}$ is the pullback of the standard volume form
$\omega_{std}=dz_1\wedge \dots\wedge dz_n$).
By pulling back the coordinate vector fields in $\C^n$ we obtain, for each $j=1,\dots,N$, 
\[
V^j_1,\dots,V^j_n\in \Vect(U^j) \text{ such that } \Span\{V^j_i(x^j)\}_i=T_{x^j} X.
\]
For each fixed $j$, define $n$ vector fields on $U$ as follows: 
for $i=1,\dots,n$, let $\Theta_i^j\in \Vect(U)$ be the trivial extension of $V_i^j$ to $U$, that is, extend it as the zero field outside of $U^j$.
(Note these are are divergence free in the other case).
Consider the vector fields $\oplus \Theta_i^j$ defined on $U^1\times\cdots\times U^N\subset Y$. They span the tangent space to $y_0=(x^1,\dots,x^N)$: 
\[
T_{y_0}Y=\Span\{\oplus\Theta_i^j(y_0)\}_{i,j}.
\]
Since $U$ is Runge in $X$, there exists $\eta^j_i\in\Vect(X)$ approximating $\Theta^j_i$ on $\overline{U}$.
Similarly in the volume case, a field $\Theta\in\Vect(U,\omega)$ over a Runge open set with $\mathrm{H}_{DR}^{n-1}(U)=0$ can be approximated by a global field $\eta\in\Vect(X,\omega)$, 
as seen in the discussion in Section~\ref{sect2.1}.
This implies 
that $\oplus \eta_i^j$ approximates $\oplus \Theta_i^j$, so we can assume that
\[
T_{y}Y=\Span\{\oplus \eta_i^j(y)\}_{i,j}
\]
holds for all $y$ in a neighborhood of $y_0$ in $Y$. By the density property, we can further approximate each $\eta_i^j$ by a finite \emph{sum} of complete vector fields $\eta_i^{j,k}$ on $X$. 
Indeed, given \emph{complete} fields $V,W\in \Vect(X)$, $[V,W]=\lim_{t\to 0^+} \frac{(V^t)^*W-W}{t}$, where $V^t$ is the time-$t$ map of the flow of $V$; 
observe that multiplication by $1/t$ and the pullback by a global automorphism preserves the completeness of a field.
Let $\eta_k\in\Vect(X)$ be the collection of the complete fields just obtained. Then the complete fields $V_k=\oplus\eta_k\in\Vect(Y)$ span $T_{y}Y$ for all $y$ in a neighborhood of $y_0$.
Exactly the same holds in presence of the $\omega$-VDP.

We now enlarge this family in order to generate the tangent spaces at any $y\in Y$. 
Notice that the fields $V_k$ span $T_y Y$ on $Y$ minus a proper analytic set $A$, which we decompose into its (possibly countably many) irreducible components $A_i \ (i\geq 1)$. 
It suffices to show that there exists $\Psi\in\Aut(X)$ (resp. $\Aut(X,\omega)$) such that $(\oplus\Psi)(Y\setminus A)\cap A_i\neq \emptyset$ for all $i$. 
Indeed, this would imply that the family $\{(\oplus\Psi)_*(V_k)\}_{k}$ of complete vector fields spans $T_{a_i} Y$ (where $a_i\in (\oplus\Psi)(Y\setminus A)\cap A_i$) for each $i$, 
so the enlarged finite collection $\{\oplus\Psi_*(V_k)\}_{k}\cup\{V_k\}_{k}$ of complete fields would fail to span the tangent space in an exceptional variety of lower dimension.
The conclusion follows from the finite iteration of this procedure. 

 To obtain this automorphism, consider \[B_i=\{\Psi\in\Aut(X); \oplus\Psi(Y\setminus A)\cap A_i\neq \emptyset \}.\]
Each $B_i$ is clearly an open set. To verify that it is also dense, let $\alpha\in\Aut(X)$ and $y^* \in A_i$. 
As above, there are finitely many complete fields $\Theta_k$ on $X$ such that $\oplus \Theta_k$ span the tangent space of $Y$ at $y^*$.
 So there is some complete $V\in \Vect(X)$ such that $\oplus V$ is not tangent to $A_i$. Thus $V^t\circ \alpha$ is an element in $B_i$ for small t (where $V^t$ is the flow of the field $V$).
 By the Baire category theorem\footnote
 {
 This applies here because $\Aut(X)$ (resp. $\Aut(X,\omega)$) is a complete metric space: see \cite[\S4.1]{KK-state}.
 } 
 there exists $\Psi\in\bigcap B_i$ and we are done.
\end{proof}
The conclusion is obviously false for the $\omega$-VDP and $dim(X)=1$: the only divergence free vector fields on $(\C,dz)$ are constant. 
%

We have just showed that there exist finitely many complete vector fields on $Y$ spanning the tangent space everywhere. 
This provides a dominating spray on $Y$, and so $Y$ is an ``elliptic'' manifold. 
It is a theorem of Gromov that such manifolds enjoy the Oka properties (see \cite{F} for details).
Hence Theorem~\ref{YisOka} is proved.

Let us derive another easy but important consequence of the above lemma.
Given complex manifolds $M$ and $X$, we say that a map $\Psi:M\to \Aut(X)$ is holomorphic if the mapping $M\times X\to X$ given by $(m,x)\mapsto \Psi(m)(x)$ is holomorphic. 
\begin{lemma}\label{lemma:smallauto}
Let $X$ be a Stein manifold with the DP (or $\omega$-VDP and $dim(X)>1$) and fix a metric $d$ on it. Let $y_0=(x^{1},\dots,x^{N})\in Y$,
$\epsilon>0$ and a compact $K\subset X$ containing each $x^j$ be given.
Then there is a neighborhood $U_\epsilon$ of $y_0$ in $Y$ with the following property: 
given a complex manifold  $W$ and an analytic homotopy  $f:W\times D_r \to Y\ (r>1)$ satisfying
\[f_t(W)\subset U_\epsilon \text{ for all } t\in D_r,\]
there exists a holomorphic map $\Psi:D_r\to \Aut_W(X)$ (or $\Aut_W(X,\omega)$)
such that $\Psi_0^w=id_X$, and such that for all $ (w,t)\in W\times D_r$,
\begin{enumerate}
\item $d(\Psi_t^w,id)<\epsilon$ and $d((\Psi_t^w)^{-1},id)<\epsilon$ on $K$;
\item and $(\oplus\Psi_t^w)\circ f_0(w)=f_t(w)$.\qedhere
\end{enumerate}
\end{lemma}
\begin{proof}
By the previous lemma, there are complete (divergence free) vector fields $V_1,\dots,V_m$ on $X$ such that $\{\oplus V_j(y_0)\}_j$ span $T_{y_0}Y$.
By discarding linearly dependent elements of the generating set, we can assume that $m=n N$, where $n$ is the dimension of $X$.
Let $\phi_j$ be the flow of $V_j$. By completeness its time-$t$ map, denoted $\phi_j^t$, is a (volume-preserving) automorphism of $X$.
Define two holomorphic maps $\phi,\phi_{-}: \C^m\times X\to X$ by 
\[
\phi(\mathbf{t},z)=\phi_1^{t_1}\circ\dots\circ \phi_m^{t_m}(z) \quad
\phi_{-}(\mathbf{t},z)=\phi_m^{-t_m}\circ\dots\circ \phi_1^{-t_1}(z).\]
and consider the holomorphic map $\varphi:\C^m\to \Aut(X)$ (or $\Aut(X,\omega)$)
given by 
\[\varphi(\mathbf{t})=\phi(\mathbf{t},\cdot):X\to X;\]
define $\varphi_{-}$ analogously.
By continuity there exists
a ball $B_R\subset \C^m$ around $\mathbf{0}$ such that for each $\mathbf{t}\in B_R$,
\[
d(\varphi(\mathbf{t}),id)<\epsilon/2 \text{ and } d(\varphi_{-}(\mathbf{t}),id)<\epsilon/2 \text{ on } K^\epsilon,
\]
where $K^\epsilon$ is a compact containing the $\epsilon$-envelope $\set{x\in X; d(x,K)<\epsilon}$ of $K$.
Consider now the map $s:\C^m\to Y$ defined by
\[
s(\mathbf{t})=(\phi(\mathbf{t},x^1),\dots,\phi(\mathbf{t},x^N)).
\]
Then $s(\mathbf{0})=y_0$ and, for all $j=1,\dots,m$,
\[
\frac{\partial s}{\partial t_j}(\mathbf{0})=(V_j(x^1),\dots,V_j(x^N))=\oplus V_j(y_0).
\]
Since $\Span\{\oplus V_j(y_0)\}_j=T_{y_0}Y$, by the implicit function theorem $s$ is locally biholomorphic on a neighborhood (which we assume contained in $B_R$) of $\mathbf{0}$ onto 
a neighborhood $U_\epsilon$ of $y_0$ in $Y$.
Then the holomorphic mapping $\psi=\varphi\circ s^{-1}:U_\epsilon\to \Aut (X)$ (or $\Aut(X,\omega)$) clearly satisfies, for each $y\in U_\epsilon$, $(\oplus\psi(y))({y_0})={y}$, and  
\[
d(\psi(y),id)<\epsilon/2 \text{ and } d(\psi^{-1}(y),id)<\epsilon/2 \text{ on }K^\epsilon.\]
Now set
\[
\tilde{\Psi}_t^w(x)=\psi(f_t(w))(x)
\]
and define $\Psi:D_r\to\Aut_W(X)$ (or $\Aut(X,\omega)$) by
$
\Psi_t^w=\tilde{\Psi}^w_{t}\circ\left(\tilde{\Psi}_0^w\right)^{-1}.
$
\end{proof}
We call such a map $\Psi:D_r\to \Aut_W(X)$ satisfying $\Psi_0^w=id$ an \emph{analytic isotopy of parametrized automorphisms}. 
Let us point out two consequences that will be of use.
In the first place, note that an analogous result holds if $f_t$ is a homotopy with $t$ varying smoothly in $\I$ instead of a complex disc:
we obtain a \emph{smooth} 
isotopy of parametrized automorphisms $\Psi:\I\to \Aut_W(X)$ satisfying the same properties. 
As a second remark, observe that the map $s$ in the above proof is defined independently of $\epsilon$, and hence so is $U$.
Therefore, if $W$ is compact and a single map $f:W\to Y$ satisfies $f(W)\subset U'\subset U$, since $s^{-1}(f(W))$ is compact in $B_{R'}\subset B_R$, for $\eta>0$ small enough and $r=1+\eta$ the function 
\[
S_t(w)=s({t}\cdot(s^{-1}\circ f(w))),\quad (w,t)\in W\times D_{r}
\]
takes values in $U'$ and defines an analytic homotopy between the constant $S_0=y_0$ and $S_1=f:W\to Y$. 
We end this section with a corollary of Theorem~\ref{YisOka}.
 \begin{corollary}\label{Oka+smooth}
Let $W$ and $X$ be as in Theorem~\ref{thm:main} or~\ref{thm:VDP}.
Then any two holomorphic maps $f_0,f_1:W\to Y_{X,N}$ which are homotopic are smoothly homotopic through holomorphic maps.
 \end{corollary}
 \begin{proof}
Let $Y$ be any Oka-Forstneri\v c manifold.
We prove that if $f:W\times\I\to Y$ is a homotopy between two holomorphic maps $f_0$ and $f_1$, then they are in fact homotopic via an analytic homotopy, so
in particular they are smoothly homotopic through holomorphic maps.

Let $r>1$ and $R:D_r\to \I\subset\C$ be any continuous retraction of the disc $D_r\subset\C$ onto the interval. Then
\begin{align*}
F:W\times D_r &\to Y\\
(w,t)&\mapsto f_{R(t)}(w)
\end{align*}
is a continuous map extending $f$ from $W\times \I$ to $W\times D_r$. Now $T=W\times \partial \I$ is a closed complex submanifold of the Stein manifold $S=W\times D_r$.
The map $F$ is holomorphic when restricted to $W\times \partial\I$, so according to the Basic Oka Property with interpolation (but no approximation) it can be deformed 
to a holomorphic map $H:W\times D_r\to Y$, which equals $F$ on $W\times\partial \I$.
\end{proof}
Note that this proof does not allow to obtain additionally approximation over a $\hol(W)$-convex compact piece $L$.
Compare with Corollary~\ref{stepchange} below.
\section{Proof of the Main Theorems}\label{sect4}
We will prove Theorem~\ref{thm:VDP}. A similar and simpler proof for Theorem~\ref{thm:main} can be extracted, by ignoring the complications arising from the preservation of the volume form.
We first go through some technicalities to prepare for the proof.

Let $X,Y$ and $W$ be as in Theorem~\ref{thm:VDP}. Fix from now on a distance function $d$ on $X$, an $N$-tuple $\hat{x}=(\hat{x}^1,\dots,\hat{x}^N)\in Y$, 
a holomorphic map $x_0=(x^1,\dots,x^N):W\to Y$, and a homotopy $x:W\times \I \to Y$ between $x_0$ and $x_1=\hat{x}$. 
The metric $d$ induces a natural distance in $Y$: for $\eta,\zeta\in Y$, let \[d_Y(\eta,\zeta)=\max_{j=1,\dots,N}d(\eta^{j},\zeta^{j}).\]

We will now construct automorphisms $\alpha_j\in\Aut_W(X,\omega)$ and verify that they converge to an element in $\Aut_W(X,\omega)$. 
For this we apply the following criterion, which appears in \cite[Prop. 5.1]{Finterpol} for $X=\C^n$ and $W=\{\emptyset\}$.
\begin{lemma}\label{lem:compo}
Let $X$ be a Stein manifold with metric $d$ and $W$ be any manifold. Suppose $W$ is exhausted by compact sets $L_j$ ($j\geq 1$), and $X$ by compacts $K_j$ ($j\geq 0$). 
For each $j\geq 1$, let $\epsilon_j$ be a real number such that
\[0<\epsilon_j< d(K_{j-1},X\setminus K_j) \  \text{ and }\  \sum \epsilon_j<\infty.\]
For each $j\geq m \geq 1$, let $\alpha_j\in\Aut_W(X)$, and let $\beta_{j,m}^w\in \Aut(X)$ be defined by
\[\beta_{j,m}^w=\alpha_j^w\circ\dots\circ\alpha_m^w.\]
Assume that for each $w\in L_j\setminus L_{j-1}$ (take $L_0=\emptyset$),
\begin{align}
d(\alpha^w_j,id) & <\epsilon_j && \text{ on } K_j\label{control}\\
d(\alpha^w_{j+s},id)& <\epsilon_{j+s} && \text{ on } K_{j+s}\cup\beta^w_{j+s-1,j}(K_{j+s})\quad \forall s\geq 1\label{control2}.
\end{align}
Then
$\beta=\lim_{m\to\infty}\beta_{m,1}$ exists uniformly on compacts and defines an element in $\Aut_W(X)$, or in $\Aut_W(X,\omega)$ if each $\alpha_j\in\Aut(X,\omega)$.
\end{lemma}
\begin{proof}
Let $w\in L_1$. The remark which is the content of \cite[Prop. 1]{Tyson} shows that if (\ref{control}) holds for all $j$, 
then the limit $\beta^w$ is injective holomorphic map onto $X$ defined on the set which consists exactly of the points $z$ in $X$ such that the sequence $\set{\beta_{m,1}^w(z);m\in \N}$ is bounded.
If we assume furthermore that 
\[
d(\alpha_{s}^w,id)<\epsilon_{s} \text{ on } K_{s}\cup\beta_{s-1,1}^w(K_{s})\quad \forall s\geq 2,
\]
which is equation (\ref{control2}) for $j=1$, we can ensure that the set of convergence for $\beta^w$ is $X$. 
Hence $\{\beta^w_{m,1}\}_m$ converges to an automorphism of $X$ if $w\in L_1$.
For $w\in L_j\setminus L_{j-1}$ and $j\geq 2$, 
the same reasoning shows that $\lim_{m\to\infty}\beta^w_{m+j,j}$ is an automorphism and we obtain $\beta^w\in\Aut(X)$ by precomposing it with the automorphism $\beta_{j-1,1}^w$.
It is clear from the construction that $\beta$ depends holomorphically on $w$, since the convergence is uniform on compacts.
\end{proof}
In practice we will construct the automorphisms $\alpha_j$ for $j\geq 1$ inductively. Observe that when defining $\alpha_j$, there are only $j$ constraints to satisfy: 
$d(\alpha_j^w,id)<\epsilon_j$ should hold 
\begin{itemize}
\item on $K_j$ if $w\in L_j\setminus L_{j-1}$, according to equation (\ref{control});
\item on $K_j\cup\beta^w_{j-1,m}(K_j)$ if $w\in L_m\setminus L_{m-1} (1\leq m\leq j-1)$, according to (\ref{control2}).
\end{itemize}
By Corollary~\ref{Oka+smooth}, we can assume that $x_0$ and $\hat{x}$ are smoothly homotopic through holomorphic maps $x_t:W\to Y$, so 
by Proposition~\ref{isotopy1} (see remarks preceding its proof) we could obtain $\alpha\in\Aut_W(X,\omega)$ mapping $x_0$ close to $\hat{x}$ over some $L\subset W$. 
Over $L$ we have a ``small homotopy'' which sends $\Gamma_{L}(\oplus\alpha\circ x_0)$ to $\hat{x}$ and on the rest of $W$ some homotopy is given.
So Proposition~\ref{isotopy1} should instead be applied to some motion coming from the ``glueing'' of these homotopies, whose holomorphic dependence on $w$ relies on the Oka property.
Its existence follows from this technical lemma.
\begin{lemma}\label{connecting2}
Let $L$ be a $\hol(W)$-convex compact set, and $f_t:W\to Y$ be a smooth homotopy between some holomorphic map $f_0$ and the constant $f_1=\hat{x}$. 
  Then there exists an $\varepsilon>0$ depending on $f$ and $L$ with the following property: 
 for every $\varepsilon'\leq \varepsilon$, every smooth $F:W\times\I\to Y$ with $F_t=f_{2t-1}$ for $t\geq 1/2$ satisfying
\begin{equation}\label{close}
d_Y(F_t(w),F_{1-t}(w))<\varepsilon'/2 \quad \forall (w,t)\in L\times [0,1],
\end{equation}
and every  $\hol(W)$-convex compact $L^-\subset\operatorname{int}(L)$, there exists an analytic homotopy $H:W\times D_r\to Y$ between $F_0$ and $\hat{x}$ such that
\[
d_Y(H_t(w),\hat{x})<\varepsilon'\quad \forall (w,t)\in L^-\times D_r.\]
\end{lemma}
\begin{proof}
The injectivity radius for the metric $d_Y$ is bounded from below by a positive constant on the compact $f(L\times \I)$. 
We let $\varepsilon$ be the minimum of this constant and of the radius (in the metric $d_Y$) of the open set $U$ mentioned in the second remark following Lemma~\ref{lemma:smallauto}. 
Fix $\varepsilon'\leq\varepsilon$ and let $F:W\times\I:\to Y$ be as above. Then $F_0(L)\subset Y$
lies in a certain $B_{\varepsilon'/2}\subset U$, so according to that remark 
there is an analytic homotopy $S:L\times D_R\to Y$ between $S_0=F_0$ and $S_1=\hat{x}$ satisfying
\[
d_Y(S_t(w),\hat{x})<\varepsilon'/2\quad \forall (w,t)\in L\times D_R.
\]
Denote by $\sigma$ the restriction of $S$ to $L\times\I$. 
We claim that there
is a continuous $\mathbf{h}:L\times\I_s \times \I_t \to Y$ such that 
\begin{center}
$\begin{array}{ll}
\mathbf{h}(w,0,t)=F_{t}(w) & \mathbf{h}(w,s,0)=\sigma_0(w)\\
\mathbf{h}(w,1,t)=\sigma_t(w) & \mathbf{h}(w,s,1)=\sigma_1(w).
\end{array}$            
\end{center}
Consider $w\in L$ fixed. By the definition of $\varepsilon$, for each $s\in\I$ there is a unique geodesic path in $Y$ from $F_s(w)$ to $F_{1-s}(w)$. By following it at constant speed, 
the parametrization $\gamma^w_s:\I_t\to Y$ is uniquely determined. Let $h^w:\I_s\times\I_t\to Y$ be defined by
\begin{align*}
 h^w_s(t)=\begin{cases}
        F_{t}(w) &\text{ if }0\leq t\leq s/2 \\
	\gamma^w_{s/2}(l(t)) &\text{ if } s/2 \leq t\leq 1-s/2 \\
	F_t(w) &\text{ if } 1-s/2\leq t\leq 1,
        \end{cases}
\end{align*}
where $l$ is the linear function of $t$ taking values $0$ at $s/2$ and $1$ at $1-s/2$. This is a well-defined homotopy between $F$ and the geodesic segment $h_0^w$ going from $F_0(w)$ to the constant $\hat{x}$;
 it is uniquely defined for each $w$.
By letting $w$ vary in $L$, all the elements in the definition of $h_s(t)$ vary continuously, so $h:L\times\I\times\I\to Y$ provides a homotopy of homotopies between $F$ and the geodesic segment $h_0$.
Now it suffices to connect $\sigma:L\times \I\to Y$ to $h_0:L\times\I\to Y$ and compose; this is achieved in a similar way, and the claim is proved.
 
Let $\rho:D_R\times \I \to D_R$ be a homotopy between the identity $\rho_0$ and a continuous retraction $\rho_1:D_R\to \I$, 
and extend $\mathbf{h}$ to $L\times \I_s\times D_R$ by defining
  \begin{align*}
     \mathbf{H}(w,s,t)=
    \begin{cases}
\mathbf{h}(w,2s,\rho_1(t)) &\text{if } 0\leq s \leq 1/2\\
S(w,\rho_{2-2s}(t)) &\text{if } 1/2\leq s\leq 1.
    \end{cases}
\end{align*}
Let $U$ be a neighborhood of $L^-$ such that $\overline{U}\subset\operatorname{int}(L)$. 
Then there exists a smooth function $\chi:W\to\I$ such that $\chi|_{U}=1$ and $\chi|_{W\setminus L}=0$.
Define $\tilde{H}:W\times D_R \to Y$ by 
  \begin{align*}
     \tilde{H}(w,t)=
    \begin{cases}
      \mathbf{H}(w,\chi(w),t) & \text{if } w\in L,\\
      F_{\rho_1(t)}(w) & \text{if } w\notin L.
    \end{cases}
\end{align*}
Consider the inclusion of the closed complex submanifold $T=W\times \partial\I$ into the Stein manifold $\mathfrak{S}=W\times D_R$. 
The map $\tilde{H}$ is continuous on $\mathfrak{S}$, restricts to the holomorphic maps $\sigma_0$ and $\sigma_1$ on $T$, and is equal to the holomorphic $S$ on a neighborhood of the $\hol(\mathfrak{S})$-convex set $L^-\times \overline{D_{r}}$ (for some $1<r<R$). 
By the Oka Property, there is a holomorphic map ${H}:\mathfrak{S}\to Y$ which restricts to $\tilde{H}$ on $T$ and approximates $\tilde{H}$ on $L^-\times {D_{r}}$.
 \end{proof}
Let us illustrate how we will use this lemma.
\begin{corollary}\label{stepchange}
Let $f:W\times \I\to Y$ be a smooth homotopy through holomorphic maps connecting $f_0$ to some constant $f_1=\hat{x}$. 
Given a $\hol(W)$-convex compact $L$ and $\epsilon>0$ small enough,
there exists $\alpha\in\Aut_W(X,\omega)$ and a smooth homotopy through holomorphic maps $h:W\times \I\to Y$ with $h_0=\oplus \alpha \circ f_0$, $h_1=\hat{x}$, and
\[
d_Y(h_t(w),\hat{x})<\epsilon\quad \forall (w,t)\in L\times \I.
\]
\end{corollary}
\begin{proof}
Pick a $\hol(W)$-convex compact $L^+$ such that $L\subset \operatorname{int}(L^+)$ and let $0<\epsilon<\varepsilon(f,L^+)$ where $\varepsilon$ is as in the lemma.
By Proposition~\ref{isotopy1}, there is $A_t\in\Aut_W(X,\omega)$ depending smoothly on $t$ such that $A_0=id$ and 
\[d_Y(\oplus A_t^w\circ f_0(w),f_t(w))<\epsilon/2 \quad \forall (w,t)\in L^+\times\I.\]
Let $\alpha=A_1$. Define $F:W\times \I\to Y$ by
\begin{equation*}
F_t(w) = 
    \begin{cases}
      \oplus A_{1-2t}^w\circ f_0(w) & \text{if } t \leq 1/2\\
      \ f_{2t-1}(w) & \text{if } t \geq 1/2.
    \end{cases} 
\end{equation*}
This is a smooth homotopy between the holomorphic map $\oplus \alpha\circ f_0$ and $\hat{x}$. 
By the above inequality $F_t$ satisfies (\ref{close}), 
so the lemma yields the desired homotopy by restricting $H$ to $\I$.
\end{proof}
We now prove the main technical tool, which roughly said allows us to iterate the approximations over a growing sequence of compacts in $W$.
\begin{proposition}\label{tricky}
Let $\eta>0$ and $K$ a compact in $X$ containing each $\hat{x}^j$  be given. Then there exists a real number $\delta(K,\eta)>0$
with the following property. If $h:W\times \I \to Y$ is a smooth homotopy  through holomorphic maps, with $h_1=\hat{x}$ and approximation
\[
d_Y(h_t(w),\hat{x})<\delta(K,\eta)\quad \forall (w,t)\in L_1 \times\I,
\]
where $L_1\subset W$ is a $\hol(W)$-convex compact,  then: \\
(a) There exists a smooth isotopy of parametrized automorphisms $\Psi:\I\to \Aut_{L_1}(X,\omega)$, such that for all $(w,t)\in L_1\times \I$,
\begin{align}
 \oplus \Psi_t^w\circ h_0(w)&=h_t(w),\nonumber\\
d(\Psi_t^w,id)&<\eta \text { on } K.\label{control4}
\end{align}
(b) Given $\epsilon>0$, $L_2$ a $\hol(W)$-convex compact containing $L_1$, and a $\hol(X)$-convex compact $C$, there exists a smooth isotopy 
$A_t\in\Aut_W(X,\omega)$ with $A_0=id$ such that
\begin{equation}\label{co1}
d(A_t^w(z),\Psi_t^w(z))<\eta \quad \forall (w,z,t)\in L_1\times C\times \I 
\end{equation}
and \[d_Y(\oplus A_t^w\circ h_0(w),h_t(w))<\epsilon \text{ on } L_2\times\I.\]
\end{proposition}
\begin{proof} 
(a) The existence of $\delta(K,\eta)$ and the volume-preserving $\Psi_t$ with these properties follows immediately from the first remark following Lemma~\ref{lemma:smallauto}.\\
(b) Define a time-dependent vector field on $L_1\times X$ by
\[
\Theta_s^w(x)=\frac{d}{dt}\bigg|_{t=s}\Psi_t^w\left( (\Psi_s^w)^{-1}(x)\right).
\]
It satisfies, for each $j=1,\dots,N$ and $s\in \I$,
\[
\Theta_s^w(h_s^j(w))=\frac{d}{d t}\bigg|_{t=s}h_t^j(w),
\]
which implies that $\Theta_s^w(x)$ is a vector field on $(L_1\times X)\cup \Gamma_W(h_s)$. 
We will show that, for each $s$, this field can be extended to a neighborhood of $(L_1\times X)\cup \Gamma_W(h_s)$ with approximation on $L_1\times C$.

There is a smooth isotopy of parametrized automorphisms $\beta_t\in\Aut_{L_2}(X,\omega)$ such that $\beta_1=id$ and $(\oplus \beta_t)\circ h_t=\hat{x}$. Indeed, 
 Proposition~\ref{isotopy1} applied to $h_t$ provides $\tilde{B}_t\in\Aut_W(X,\omega)$, depending smoothly on $t$, with the property that 
${B}_t=\tilde{B}_1\circ \tilde{B}_t^{-1}$ maps  $\Gamma_{L_2}(h_{t})$ arbitrarily close to  $\Gamma_{L_2}(\hat{x})$. Hence Lemma~\ref{lemma:smallauto} applied to $\oplus B_t \circ h_t:W\to Y$ gives
elements $\Phi_t\in\Aut_{L_2}(X,\omega)$, depending smoothly on $t$,
such that
\[
\oplus (\Phi_t^w \circ  B_t^w)\circ h_t(w)=\oplus B_0^w\circ h_0(w)\quad \forall w\in L_2.
\]
Then ${\Phi_1}^{-1}\circ\Phi_t\circ  B_t\in\Aut_{L_2}(X,\omega)$ is the desired $\beta_t$.\\
The pushforwards $(\beta_t)_*(\Theta_t)$ define together a divergence free time-dependent vector field on $(L_1\times X)\cup \Gamma_{L_2}(\hat{x})$. 
Just as in the proof of Proposition~\ref{isotopy1}, this can be extended from the analytic subvariety $\Gamma_{L_2}(\hat{x})$ to a neighborhood of it, and moreover it is a 
classical result of E. Bishop \cite{Bishop} following from Cartan's theorem A and B that this can be done with smooth dependence on the $t$ parameter and with
arbitrary approximation on a large $\hol(W\times X)$ compact of the form $L_1\times \tilde{K}$, where $\tilde{K}$  contains
\begin{equation}\label{cupnotation}
\beta_{\I}^{L_1}(C)=\set{\beta_t^w(x); w\in L_1, x\in C, t\in\I}\subset X .
\end{equation}
Its pullback is an approximate extension of the time-dependent vector field $\Theta$ above, whose flow provides
 an isotopy of injective
volume-preserving  holomorphic maps $F_t:\Omega\to W\times X$, 
where $\Omega$ is a neighborhood of $\Gamma_{L_2}(h_0)$ containing $L_1\times X$, and
 such that 
\begin{align}
d(F_t^w(z),\Psi_t^w(z))<\eta/2\quad &\forall (w,z,t)\in L_1\times C\times \I \label{t1}\\
\oplus F_t^w \circ h_0(w)=h_t(w) \quad  &\forall  (w,t)\in L_2\times\I.\label{t2}
\end{align}
Observe that in fact $\Psi^w$ is defined for $w$ in a neighborhood of $L_1$, so we may apply Proposition~\ref{isotopy2}. We obtain $A_t\in\Aut_W(X,\omega)$ such that
\[
d(A_t^w(z),F_t^w(z))<\min(\epsilon,\eta/2)
\]
on $(L_1\times C) \cup \Gamma_{L_2}(h_0)$. This and (\ref{t1}) show that (\ref{co1}) holds. Furthermore, by (\ref{t2}), 
\[
d_Y(\oplus A_t^w \circ h_0(w),h_t(w))<\epsilon \quad \forall(w,t)\in L_2\times\I.  \qedhere
\]
\end{proof}
\begin{proof}\textit{(of Theorem~\ref{thm:VDP})}
The ``only if'' part follows from the  definition of the path connected component; we have to prove the ``if'' part. 
Fix a compact exhaustion of $W\times X$, of the form $W=\bigcup_{j=1}^\infty L_j$ and $X=\bigcup_{j=0}^\infty K_{j}$,
 where each $L_j$ (resp. $K_j$) is a $\hol(W)$-convex (resp. $\hol(X)$-convex) compact set, and such that $L_j\subset\operatorname{int}(L_{j+1})$. 
Fix also real numbers $\epsilon_j$ 
($j\geq 1)$ such that $0<\epsilon_j<d(K_{j-1},X \setminus K_{j})$ and $\sum \epsilon_j<\infty$. 
We can suppose that $K_0$ contains $\hat{x}^{j}$ for all $j=1,\dots,N$.

By Corollary~\ref{Oka+smooth}, $x_0$ and $\hat{x}$ are smoothly homotopic through holomorphic maps. Hence Corollary~\ref{stepchange} gives
$\alpha_0\in\Aut_W(X,\omega)$ and a smooth homotopy of holomorphic maps $h:W\times \I\to Y$ between $h_0=\oplus \alpha_0\circ x_0$ and $h_1=\hat{x}$ with
\[
d_Y(h_t(w),\hat{x})<\delta(K_1,\epsilon_1/2) \quad\forall (w,t)\in L_1\times \I,
\]
where $\delta>0$ is as in Proposition~\ref{tricky}. Apply part (a) of it to $h_t$: we obtain some $\Psi:\I\to\Aut_{L_1}(X,\omega)$. Consider the compact
\[\Psi_{\I}^{L_1}(K_2)\]
(recall the notation from equation (\ref{cupnotation})) and define $C_1$ to be a $\hol(X)$-convex compact containing its $(\epsilon_1/2)$-envelope.
By part (b) of Proposition~\ref{tricky}, we obtain a smooth isotopy of automorphisms $A_t\in\Aut_W(X,\omega)$ with $A_0=id$ such that
\[
d_Y(\oplus A_t^w \circ h_0(w),h_t(w))<\min(\epsilon_1,\delta(C_1,\epsilon_2/2),\varepsilon(h,L_3))/2\quad\forall (w,t)\in L_3\times \I, 
\]
where $\varepsilon$ is as in Lemma~\ref{connecting2}.
Combining (\ref{control4}) and (\ref{co1}) shows
\[
d(A_t^w(z),z)<\epsilon_1 \quad\forall (w,z,t)\in L_1\times K_1\times\I.
\]
We let $\alpha_1=A_1$. Then in particular
\[
d(\alpha^w_1,id)<\epsilon_1 \text{ on }L_1\times K_1.
\]
Thus $\alpha_1$ satisfies the only condition imposed by (the remark following) Lemma~\ref{lem:compo}. Observe finally that by (\ref{co1}), $\alpha_1^w(K_2)\subset C_1$ for $w\in L_1$.

We now construct inductively $\alpha_j$ for $j\geq 2$.  
Fix $k\geq 1$ and assume that we have defined $C_j\subset X$ and $\alpha_j\in \Aut_W(X,\omega)$, for all $1\leq j\leq k$, 
such that the following conditions hold (recall that $\beta_{j,m}^w=\alpha_j^w\circ\dots\circ\alpha_m^w$ and $\beta_j=\beta_{j,0}$):
\begin{itemize}
\item[(a)] $\alpha_j$ is smoothly isotopic to the identity through some $A_t\in\Aut_W(X,\omega)$; 
\item[(b)] $d_Y(\oplus A_t^w \circ h_0(w),h_t(w)) < \min(\epsilon_j,\delta (C_j,\epsilon_{j+1} /2),\varepsilon(h,L_{j+2}))/2$ for all $(w,t)\in{L_{j+2}}\times\I$,
where $h:W\times \I\to Y$ is a smooth homotopy between $\oplus \beta_{j-1}\circ x_0$ and $\hat{x}$;
\item[(c)] $C_j$ contains $K_{j+1}$, and $\set{\beta^w_{j,m}(K_{j+1});w\in L_m\setminus L_{m-1}}$ for every $1\leq m\leq j$;
\item[(d)] and every $A_t^w$ satisfies the $j$ conditions of Lemma~\ref{lem:compo}, that is, for every $1\leq m\leq j$, if $w\in L_m\setminus L_{m-1}$, then $d(A^w_t,id)<\epsilon_j$ on $K_j\cup \beta^w_{j-1,m}(K_j)$.
\end{itemize}
We have just verified that these conditions hold for $k=j=1$. Let $j\geq 1$.
It suffices to show that $\alpha_{j+1}$ and $C_{j+1}$ can be constructed satisfying the above conditions: 
indeed, by condition (d), Lemma~\ref{lem:compo} would imply that $\beta=\lim_{j\to \infty} \beta_{j,1}\in \Aut_W(X,\omega)$ exists, 
and by construction (since $\epsilon_j\to 0$) $\oplus\beta$ maps $\alpha_0\circ x_0$ to $\hat{x}$, so $\beta\circ \alpha_0\in \Aut_W(X,\omega)$ would be the simultaneous standardization.
Further, by conditions (a) and (d), $\beta\circ\alpha_0$ lies in  $(\Aut_W(X,\omega))^0$. 

So let $A$ and $h$ be as in conditions (a) and (b) at step $j$. By the inequality in condition (b), and since $L_{j+1}\subset\operatorname{int}(L_{j+2})$, we can apply Lemma~\ref{connecting2} to 
  \begin{align*}
     {F}_t(w)=
    \begin{cases}
      \oplus A_{1-2t}^w \circ h_0(w) &\text{if } t\leq 1/2\\
      h_{2t-1}(w) &\text{if } t\geq 1/2.
    \end{cases}
\end{align*}
We obtain a smooth homotopy through holomorphic maps $H:W\times\I\to Y$, such that $H_0=\oplus\beta_j\circ x_0$ and $H_1=\hat{x}$ and for all $t\in\I$,
\[
d_Y(H_t(w),\hat{x})<\delta(C_{j},\epsilon_{j+1} /2)\quad \forall w\in L_{j+1}.
\]

By the first part of Proposition~\ref{tricky} there is a smooth isotopy
\[
\Psi:\I\to \Aut_{L_{j+1}}(X,\omega)
\]
with $\oplus \Psi_t^w\circ H_0(w)=H_t(w)$ and 
\begin{equation}\label{I}
d(\Psi_t^w,id)<\epsilon_{j+1}/2 \text{ on } L_{j+1}\times C_{j}.
\end{equation}
Define $C_{j+1}$ to be a $\hol(X)$-convex compact containing the $(\epsilon_{j+1}/2)$-envelope of 
\begin{equation}\label{env}
C_j\cup \Psi_{\I}^{L_{j+1}}(K_{j+2})\cup \bigcup_{1\leq m\leq j}\Psi_{\I}^{L_m}(\beta_{j,m}(K_{j+2})).
\end{equation}
By the second part of Proposition~\ref{tricky}, there are  $A_t^w\in\Aut_{W}(X,\omega)$ smoothly depending on $t$ and with $A_0=id$ such that
\begin{align}
d(A^w_t,\Psi_t^w) &< \epsilon_{j+1}/2  \text{ on } L_{j+1}\times C_{j+1}\times\I \label{III} \\
d_Y(\oplus A_t^w\circ H_0(w),H_t(w))&<\min(\epsilon_{j+1},\delta(C_{j+1},\epsilon_{j+2}/2),\varepsilon(H,L_{j+3}))/2 \text{ on }L_{j+3} \times\I\label{IIII}.
\end{align} 
Define $\alpha_{j+1}=A_1$, so condition (a) of the induction is met at step $j+1$. Equation (\ref{IIII}) means 
that condition (b) is also satisfied.

Let us check that condition (d) holds at step $j+1$. Note that by  (\ref{III}) and (\ref{I})
\[
d(A^w_t,id)<d(A_t^w,\Psi_t^w)+d(\Psi^w_t,id)<\epsilon_{j+1} \text{ on }L_{j+1}\times C_j.
\]
By condition (c), $C_j$ contains $K_{j+1}\cup \beta_{j,m}^w(K_{j+1})$ for any $w\in L_m\setminus L_{m-1}$, where $1\leq m\leq j+1$, so
 $d(A_t^w,id)<\epsilon_{j+1}$ on $K_{j+1}$. 

It remains to show that $C_{j+1}$ satisfies condition (c). Since $\Psi_0=id$, it contains $K_{j+2}$. Let $1\leq m\leq j+1$, $w\in L_{m}\setminus L_{m-1}$  and $z\in K_{j+2}$. 
By the definition of $C_{j+1}$, it suffices to check that 
\begin{equation}\label{final}
d(\beta^w_{j+1,m}(z),z')<\epsilon_{j+1}/2
\end{equation}
where $z'$ is some element of the compact (\ref{env}). 
If $m=j+1$, pick $z'=\Psi_1^w(z)$. Then (\ref{final}) follows from (\ref{III}).
If $m<j+1$, let $z'=\Psi_1^w(\beta_{j,m}^w(z))$, which belongs to (\ref{env}). Then
\[
d(\beta^w_{j+1,m}(z),z')=d(\alpha_{j+1}(\beta^w_{j,m}(z)),\Psi_1^w(\beta_{j,m}^w(z)))<\epsilon_{j+1}/2
\]
where the inequality again follows from (\ref{III}), since $\beta_{j,m}^w(z)\in C_{j+1}$. The induction is complete.
\end{proof}

\section{Concluding Remarks}\label{last}
In this section we change slightly our point of view. With $W$ and $X$ as before, we consider $\Hol(W,Y_{X,N})$, the space of $N$ parametrized points in $X$.
We identify the group $\Aut_W(X)$ with the group of holomorphic mappings from $W$ to $\Aut(X)$, which we denote by $G=\Hol(W,\Aut(X))$. We naturally get an identification between
$G_0$, the path-connected component of the identity in $G$, with $(\Aut_W(X))^0$.
The group $G$ acts on the space $\Hol(W,Y_{X,N})$ by
\[
(\alpha\cdot x)(w)=(\oplus \alpha(w))\circ x(w)
\]
where $x=(x^1,\dots,x^N)\in\Hol(W,Y_{X,N})$ as before. It also acts on the space of homotopy classes (or path-connected components), which we denote here by $[\Hol(W,Y_{X,N})]$. 
Since the path-connected component $G_0$ of the identity in $G$ acts trivially, we get an action of $G/G_0$, the space of homotopy classes  $[\Hol(W,\Aut(X))]$ of holomorphic maps from $W$ to $\Aut(X)$, on  $[\Hol(W,Y_{X,N})]$.
Then an immediate consequence of Theorem~\ref{thm:main} can be phrased as follows.
\begin{corollary}
Any $x\in\Hol(W,Y_{X,N})$ is simultaneously standardizable if and only if $G/G_0$ acts transitively on $[\Hol(W,Y_{X,N})]$.
\end{corollary}
By Theorem~\ref{YisOka}, $Y_{X,N}$ is an Oka-Forstneri\v c manifold. 
Hence the Oka principle, or \emph{weak homotopy equivalence principle} (see e.g. \cite[5.4.8]{F}) applies: 
$[\Hol(W,Y_{X,N})]$ is isomorphic to the space of homotopy classes $[\Cont(W,Y_{X,N})]$ of continuous maps from $W$ to $Y_{X,N}$.
Thus we deduce:
\begin{corollary}
Any $x\in\Hol(W,Y_{X,N})$ is simultaneously standardizable  if and only if $G/G_0$ acts transitively on $[\Cont(W,Y_{X,N})]$.
\end{corollary}
Let us consider the special case $X=\C^n$, $n>1$. The group of holomorphic automorphisms $\Aut(\C^n)$ admits a strong deformation retract onto $\GL_n(\C)$. 
Therefore 
\begin{equation}\label{Oka1}
[\Hol(W,\Aut(\C^n))]\cong [\Hol(W,\GL_n(\C))]
\end{equation}
as well as
\begin{equation*}
[\Cont(W,\Aut(\C^n))]\cong [\Cont(W,\GL_n(\C))].
\end{equation*}
By the Oka principle (since $\GL_n(\C)$ is Oka-Forstneri\v c),
\begin{equation}\label{Oka3}
[\Hol(W,\GL_n(\C))]\cong[\Cont(W,\GL_n(\C))].
\end{equation}
As a consequence, the following purely topological characterization of simultaneous standardization can be deduced from our main theorem.
\begin{corollary}\label{hom-coro}
Any  $x\in\Hol(W,Y_{\C^n,N})$ is simultaneously standardizable if and only if $[\Cont(W,\GL_n(\C))]$ acts transitively on $[\Cont(W,Y_{\C^n,N})]$.
\end{corollary}
We also see from equations (\ref{Oka1}) to (\ref{Oka3}) that
\begin{equation}\label{A}
[\Cont(W,\Aut(\C^n))]\cong[\Hol(W,\Aut(\C^n))],
\end{equation}
which is a partial Oka principle of the infinite-dimensional manifold $\Aut(\C^n)$. We can ask the following question: is it true that for any Stein manifold $X$ with the density property, we have that
\[
[\Cont(W,\Aut(X))]\cong[\Hol(W,\Aut(X))]?
\]
In other words, is there an ``Oka theory'' for infinite-dimensional manifolds, and are the groups $\Aut(X)$ for $X$ a Stein manifold with the density property, Oka-Forstneri\v c manifolds in any such sense?
For a first study of  Oka properties of $\Aut(\C^n)$ and some of its subgroups, we refer the interested reader to \cite{infinite}.

Continuing with the case $X= \C^n$, we give another interpretation of our results which is a generalization of Grauert's Oka principle to principal bundles for certain infinite-dimensional  subgroups of $\Aut(\C^n)$.
First note that simultaneous  standardization is the same as lifting the map in the following diagram, where $(z^1,z^2, \dots , z^N)$  is a fixed $N$-tuple of points in $X=\C^n$:
\begin{diagram}
     &   & \Aut(\C^n)       & \alpha \\
       & \ruDashto & \dTo     &\dMapsto \\
 W  & \rTo^x & Y_{\C^n, N} & (\alpha (z^1), \alpha (z^2,)   \ldots , \alpha (z^N)) \\
\end{diagram}
Since $Y_{\C^n, N}$ is homogeneous  under $G=\Aut(\C^n)$, we can write it as $Y_{\C^n, N}=G/ H_{n,N}$, 
where $H_{n,N}$ is the (isotropy) subgroup of $G=\Aut(\C^n)$ fixing the $N$-tuple  $(z^1,z^2, \ldots,z^N)$ of points 
in $\C^n$. The above diagram in this notation becomes
\begin{diagram}
     &   & G       \\
       & \ruDashto & \dTo>\pi     \\
 W  & \rTo^x &  G/H_{n,N}   \\
\end{diagram}
where $\pi : G \to G/H_{n,N}$ is the  natural $H_{n,N}$-principal bundle. 
The existence of a  holomorphic (resp. continuous) lift in this diagram is equivalent to the fact that the pullback bundle $P_x$ with projection $x^* (\pi): x^* (G) \to W$ (which is an 
$H_{n,N}$-principal bundle over $W$) is holomorphically (resp. topologically) trivial. 
Suppose the bundle $P_x$ is topologically trivial: then there exists $\alpha_{cont} : W \to \Aut(\C^n)$ lifting $x$. By (\ref{A}) 
there is a holomorphic map $\alpha_{hol} : W \to \Aut(\C^n)$ homotopic to $\alpha_{cont}$. It follows that $\alpha_{hol}^{-1}\circ x : W \to Y_{\C^n, N}=G/ H_{n,N}$,
defined by
\[
w\mapsto\left((\alpha_{hol}^w)^{-1}\circ x^1(w),\dots,(\alpha_{hol}^w)^{-1}\circ x^N(w)\right),
\]
is null-homotopic,
and therefore lifts by Theorem \ref{thm:main}. This shows that $x$ lifts holomorphically, i.e., the bundle $P_x$ is holomorphically trivial.
We have then proven following version of Grauert's Oka principle for principal bundles under the groups $H_{n,N}$:
\begin{corollary}
For any holomorphic map $x : W \to Y_{\C^n, N}$ from any Stein manifold $W$, the $H_{n,N}$-principal bundle $P_x$, which is the pullback by $x$ of the canonical $H_{n,N}$-principal bundle $\pi : \Aut(\C^n)  \to \Aut(\C^n) / H_{n,N}$, 
is holomorphically trivial if and only if it is topologically trivial.
\end{corollary}

We end this section with two examples. 
The first shows the difference between simultaneous standardization using automorphisms in the path-connected component of the identity $(\Aut_W(X))^0$ and 
using the whole group $\Hol(W,\Aut(X))$. In this  example the map $x\in\Hol(W,Y_{X,N})$ is not null-homotopic, so the standardization cannot be achieved by automorphisms in $(\Aut_W(X))^0$;
however standardization is possible by elements in $\Hol(W,\Aut(X))$.

The second is an example where the topological obstruction from Corollary \ref{hom-coro} does  prevent from simultaneous standardization, i.e., 
in this example $[\Cont(W,\GL_n(\C))]$ does not act transitively on $[\Cont(W,Y_{\C^n,N})]$.
\begin{example}
Let $W$ be any Stein manifold. Then any $x\in\Hol(W,Y_{\C^2,2})$ is simultaneously standardizable. 
\end{example}
\begin{proof}
Let
\begin{equation*}
x=(x_1(w),x_2(w))=\left(
\binom{z_1(w)}{\eta_1(w)},\binom{z_2(w)}{\eta_2(w)}
\right),
\end{equation*}
and define
\begin{align*}
\alpha_1^w(z,\eta)&=(z-z_1(w),\eta-\eta_1(w))\\
\alpha_2^w(z,\eta)&=\left( \begin{array}{cc}
                           z_2(w)-z_1(w) & f(w)\\
			   \eta_2(w)-\eta_1(w) & g(w)
                          \end{array}
  \right)\cdot \left(\begin{array}{c}
                z\\
		\eta
               \end{array} \right)
\end{align*}
Observe that $z_2(w)-z_1(w)$ and $\eta_2(w)-\eta_1(w)$ have no common zeros. Hence, since $W$ is Stein, Cartan's theorem B implies that there are $f,g\in\hol(W)$ such that
\[
\left( \begin{array}{cc}
                           z_2(w)-z_1(w) & f(w)\\
			   \eta_2(w)-\eta_1(w) & g(w)
                          \end{array}
  \right)\in\SL_2(\C).
\]
Hence $\alpha_1,\alpha_2\in\Aut_W(\C^2)$ and $(\alpha_2^{-1})^w\circ \alpha_1^w$ maps $x_1(w)$ to $(0,0)$ and $x_2(w)$ to $(1,0)$, which gives the simultaneous standardization. 
\end{proof}
As a consequence, by Corollary~\ref{hom-coro}, $[\Cont(W,\GL_2(\C))]$ acts transitively on $[\Cont(W,Y_{\C^2,2})]$.
In order to find an example of this form where standardization by elements of $(\Aut_W(X))^0$ is not possible, consider the special case $W=\SL_2(\C)$. 
Then there exists a non null-homotopic $x\in\Hol(\SL_2(\C),Y_{\C^2,2})$ 
which can be standardized with an element not in $(\Hol(\SL_2(\C),\Aut(\C^2)))^0$.
Indeed, the holomorphic map $\SL_2(\C)\to Y_{\C^2,2}$ given by
\[
A\mapsto \left( A\binom{1}{0},\binom{0}{0} \right)
\]
induces the identity mapping on the $3$-sphere (by projection to the first factor of $Y_{\C^2,2} $), so is not a null-homotopic map. 

\begin{example}
Let $W$ be a small (so that the map below gives pairwise different points) Grauert tube around $\SU_2$, i.e., a Stein neighborhood of $\SU_2$ in $\SL_2(\C)$ which contracts onto the $3$-sphere $\SU_2$. Then  $x\in\Hol(W,Y_{\C^2,3})$ defined by 
\[
A\mapsto \left( A\binom{1}{0},\binom{0}{0}, \binom{2}{0} \right)
\]
is not  simultaneously standardizable. 
\end{example}
\begin{proof} 
Consider the  map $\phi : Y_{\C^2,3} \to S^3 \times S^3$ given by 
\[(x_1, x_2, x_3) \mapsto (\frac{x_2-x_3} {\vert x_2-x_3\vert } ,\frac{ x_1-x_2} {\vert x_1-x_2\vert})\]
Since $W$ contracts to $\SU_2 \cong S^3$ the composition $\phi \circ x: W \to S^3 \times S^3$  gives  a map from $S^3 \to S^3 \times S^3$.  It has bidegree $(0, 1)$ and applying  any element in $ [\Hol(W,\Aut(\C^2))] \cong [\Cont(W,\GL_2(\C))]$ to it,  changes both degrees by the same amount, so the corresponding bidegree will never be $(0,0)$. Therefore no application of an element in $ [\Hol(W,\Aut(\C^2))]$ to $x$ can lead to a null-homotopic map.
\end{proof}

\end{document}